\documentclass[11pt]{amsart}
\usepackage{graphicx}
\usepackage{amssymb,bm}
\usepackage[usenames,dvipsnames]{xcolor}
\usepackage{hyperref}

\date{}

\theoremstyle{plain}
\newtheorem{theorem}{Theorem}[section]

\newtheorem{lemma}[theorem]{Lemma}

\newtheorem{proposition}[theorem]{Proposition}
\newtheorem{corollary}[theorem]{Corollary}
\newtheorem{question}[theorem]{Question}

\theoremstyle{definition}

\newtheorem{remark}[theorem]{Remark}

\newcommand{\Z}{\mathbb{Z}}

\newcommand{\R}{\mathbb{R}}

\newcommand{\La}{\Lambda}

\newcommand{\Br}{\operatorname{Br}}

\newcommand{\tb}{\operatorname{tb}}
\newcommand{\rot}{\operatorname{rot}}

\newcommand{\Symp}{\operatorname{Symp}}

\newcommand{\st}{\text{st}}

\newcommand{\dfn}[1]{{\textbf{#1}}}

\begin{document}

	\title{Exact Lagrangian fillability of 3-braid closures}
	
	\author{James Hughes}
\address{Duke University, Dept. of Mathematics, Durham, NC 27708, USA}
\email{james.hughes@duke.edu}
\author{Jiajie Ma}
\address{Duke University, Dept. of Mathematics, Durham, NC 27708, USA}
\email{jason.ma@duke.edu}
\begin{abstract} We determine when a Legendrian quasipositive 3-braid closure in the standard contact $\mathbb{R}^3$ admits an orientable or non-orientable exact Lagrangian filling. Our main result provides evidence for the orientable fillability conjecture of Hayden and Sabloff, showing that a 3-braid closure is orientably exact Lagrangian fillable if and only if it is quasipositive and the HOMFLY bound on its maximum Thurston-Bennequin number is sharp. Of possible independent interest, we construct explicit Legendrian representatives of quasipositive 3-braid closures with maximum Thurston-Bennequin number.
    
\end{abstract}


	\maketitle

\section{Introduction}
Legendrian links and their exact Lagrangian fillings are central to the study of low-dimensional contact and symplectic topology \cite{Chantraine10,HaydenSabloff,BourgeoisSabloffTraynor15}.
Following earlier works \cite{EHK,YuPan,STZ_ConstrSheaves,STWZ}, recent developments have significantly advanced the construction and classification of orientable exact Lagrangian fillings of Legendrian links in the standard contact $\R^3$ \cite{CasalsGao,CasalsZaslow,ABL22,Hughes2021,CasalsNg,CasalsWeng,CGGLSS, CasalsGao23},
especially for $0$- and $(-1)$-framed closures of positive braids; see Figure~\ref{fig: rainbow} for corresponding Legendrian representatives. Concurrently, the study of non-orientable exact Lagrangian fillings has increasingly emerged as a companion to these advances \cite{CCPRSY, capovillasearle2023newton, CSHW2023}. While the works cited above represent significant progress towards understanding exact Lagrangian fillings for certain families of Legendrian links, it is not yet known exactly which smooth link types have Legendrian representatives in $\R^3$ with exact Lagrangian fillings, orientable or not. Such a link is termed \dfn{orientably} or \dfn{nonorientably exact Lagrangian fillable}, respectively.

Orientable fillability is closely related to the hierarchy of positivity in smooth knot theory. In particular, \cite{BoileauOrevkov01} and \cite{MR1362834} together imply that an orientably fillable link is at least quasipositive, and being positive \cite{HaydenSabloff}---or almost positive in some cases \cite{Tagami19}---is a sufficient condition for fillability.\footnote{For some families of knots, positivity is also a necessary condition for orientable fillability; see \cite{LipmanSabloff} for the case of 4-plat links.} However, not all (strongly) quasipositive links are orientably fillable \cite{HaydenSabloff}. Another necessary condition for orientable fillability is that the HOMFLY polynomial bound on the maximum Thurston-Bennequin number of an orientably fillable link must be sharp \cite{HaydenSabloff}; see Section~\ref{ssec:links_background}. Hayden and Sabloff conjectured that this condition together with quasipositivity characterizes all orientably fillable links \cite[Conjecture 1.3]{HaydenSabloff}. 

In comparison, non-orientable fillability is more subtle. On one hand, the Kauffman polynomial bound on the maximum Thurston-Bennequin number of a non-orientably fillable link is sharp \cite{CCPRSY}, providing an analogue of the HOMFLY bound appearing in the orientable case. On the other hand, the relationship between non-orientable fillability and the positivity hierarchy is more obscure and no conjecture exists about sufficient conditions for non-orientable fillability. While positive links are not decomposably non-orientably fillable, there are links that realize every combination of quasipositivity and non-orientable fillability, and not all negative links are non-orientably fillable \cite{CCPRSY}. There are also examples of non-orientably fillable knots that are neither quasipositive nor negative \cite{CCPRSY}.

In this paper, we answer the exact Lagrangian fillability problem for ($0$-framed) closures of quasipositive $3$-braids. Combined with previous results, it gives a complete orientable fillability classification of closures of $3$-braids. Let $\hat{\beta}$ be the smooth $0$-framed closure of a braid $\beta\in \Br_3$. Denote by $\Delta = \sigma_1\sigma_2\sigma_1$ the Garside element of $\Br_3$. Using the Garside normal form of $\beta$, up to conjugation, we can write
\begin{equation}
    \label{eq:garside-normal}
    \beta=\sigma_1^{k_1}\sigma_2^{k_2}\dots \sigma_1^{k_{r-1}}\sigma_2^{k_r}\Delta^m,
\end{equation} 
where $k_r=0$ if and only if $m$ is odd, and otherwise $k_i\geq 2$ for all $i$. Further, $m$ is uniquely determined in this form; 
see Section~\ref{ssec:3-braids} for more details. The following result characterizes orientable fillability of $\hat{\beta}$.

\begin{theorem}\label{thm: fillability}
    The closure $\hat{\beta}$ of a $3$-braid $\beta$ is orientably exact Lagrangian fillable if and only if $\hat{\beta}$ is quasipositive and $m\geq -2$.
\end{theorem}

   \begin{center}
    \begin{figure}[h!]{ \includegraphics[width=.7\textwidth]{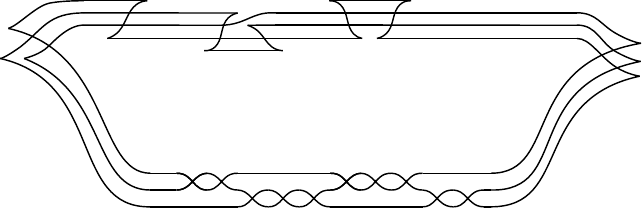}}\caption{Max-tb Legendrian representative of the closure of $\sigma_1^3\sigma_2^2\sigma_1^2\sigma_2^2\sigma_1^3\sigma_2^3\sigma_1^2\Delta^{-7}$. This knot has no orientable exact Lagrangian fillings, but is non-orientably fillable.}
    \label{fig: example_front}\end{figure}
\end{center}

By our discussion above, it suffices to consider the case when $\hat{\beta}$ is quasipositive and $m \leq -1$. The case when $m = -1$ is  covered in \cite[Section 4.3]{CGGLSS}. For $m \leq -1$, we exploit Orevkov's characterization \cite{Orevkov04} of quasipositive $3$-braids to provide explicit max-tb Legendrian representatives $\La(\beta)$ of $\hat{\beta}$; see Figure~\ref{fig: example_front} for an example. As rainbow closures of positive braids are max-tb Legendrian representatives, our construction allows us to characterize the maximum Thurston-Bennequin number of all quasipositive $3$-braids. In more detail, we have the following corollary:

\begin{corollary}\label{cor:max-tb}
    Let $\hat{\beta}$ be the closure of a quasipositive $3$-braid $\beta$ of the form given in Equation~\ref{eq:garside-normal}. The maximum Thurston-Bennequin number of $\hat{\beta}$ is given by
    \begin{equation*}
        \overline{\tb}(\hat{\beta}) = \begin{cases}
            |\beta| - 3 & m \ge -1, \\
             |\beta| + m-1 & m \leq  -2, 
        \end{cases}
    \end{equation*}
    where $|\beta|$ is the signed word length of $\beta$. 
\end{corollary}

Note that Rudolph's formula \cite[Section 3]{Rudolph93} gives a way to compute the maximal self-linking number of a quasipositive braid closure. However, in order to apply this formula to compute the maximal Thurston-Bennequin number, one needs to be able to determine the rotation number of a Legendrian representative realizing $\overline{tb}$.
When $m = -2$, we construct a decomposable orientable exact Lagrangian filling of $\La(\beta)$. When $m < -2$, we use $\La(\beta)$ to prove explicitly that the HOMFLY polynomial bound on $\overline{\tb}(\hat{\beta})$ is not sharp and thus, $\hat{\beta}$ is not orientably fillable.

Notably, when $m$ is even and less than $-2$, our construction of $\La(\beta)$ yields examples of non-fillable max-tb representatives with rotation number equal to zero.  
\begin{corollary}\label{cor:rot=0}
    There exist quasipositive max-tb Legendrian links with rotation number 0 that do not admit any orientable fillings.
\end{corollary}

Nonzero rotation number obstructs orientable fillability \cite{Chantraine10},  and to our knowledge, these are the first known examples where this obstruction vanishes but the HOMFLY bound is not sharp. A direct search of all quasipositive knots with 11 or fewer crossings via KnotInfo \cite{knotinfo} reveals only one example, namely $m(10_{155})$, which can be realized as the closure of the braid $\beta=\sigma_1^3\sigma_2^2\sigma_1^2\sigma_2^3\sigma_1^2\sigma_2^2\Delta^{-4}$.

As another corollary of Theorem~\ref{thm: fillability}, our result supports the Hayden-Sabloff conjecture.

\begin{corollary}
    The Hayden-Sabloff fillability conjecture holds for links of braid index $3$. Namely, every orientably fillable 3-braid closure is quasipositive and has sharp HOMFLY bound.
\end{corollary}

\begin{remark}
One might hope to extend our methods above to characterize orientable fillability of links with higher braid index. The primary difficulty is that we rely heavily on Orevkov's characterization of quasipositivity for 3-braids; no such characterization is known for higher braid index. To our knowledge \cite{Orevkov15} represents the furthest progress towards understanding the case of 4-braids.
\end{remark}

\begin{center}
    \begin{figure}[h!]{ \includegraphics[width=.8\textwidth]{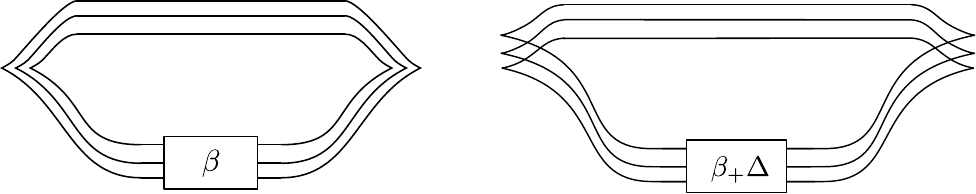}}\caption{Rainbow closure of a positive braid $\beta$ (left) and $(-1)$-closure of $\beta_+\Delta$ (right).}
    \label{fig: rainbow}\end{figure}
\end{center}

In the context of the recent work on the construction and classification of exact Lagrangian fillings of $(-1)$-closures of certain families of positive braids cited earlier, we wish to note that every orientably fillable max-tb representative we construct using our method is the $(-1)$-closure of a positive braid; see Figure \ref{fig: rainbow} (right) for a $(-1)$ closure of $\beta_+\Delta$. 
To our knowledge, this is true of
all existing examples of orientably fillable Legendrian links, including the ones studied in this paper, suggesting the following question:

\begin{question}\label{q:-1-closures?}
    Do there exist orientably fillable Legendrian links that are not $(-1)$-closures of positive braids?
\end{question}

This question has implications for the construction of Lagrangian fillings, as the method of Legendrian weaves, introduced for 2-braids in \cite{TreumannZaslow} and generalized to higher braid index in \cite{CasalsZaslow}, requires that the boundary link be given as the $(-1)$-closure of a positive braid. This method involves constructing Legendrian surfaces by combinatorially representing a certain class of Legendrian wavefront singularities, and taking the Lagrangian projection of the resulting Legendrian surface to obtain an exact Lagrangian filling. Answering Question~\ref{q:-1-closures?} in the negative would leave open the possibility that every orientably fillable Legendrian admits an exact Lagrangian filling constructed as a Legendrian weave.  A careful analysis of the fillings we construct in Section~\ref{sec: orientable-fill} reveals that this is the case for the fillings we construct in this work.

\begin{corollary}\label{cor: weaves}
Every orientably fillable 3-braid closure has at least one max-tb representative that admits an exact Lagrangian filling obtained as the Lagrangian projection of a Legendrian weave.
\end{corollary}

We believe that Corollary~\ref{cor: weaves} is of possible interest because Legendrian weaves provide a powerful tool for constructing and distinguishing exact Lagrangian fillings. Weaves are currently the best way of identifying what are known as $\mathbb{L}$-compressing cycles of a filling $L$, i.e. elements of $H_1(L)$ that bound Lagrangian disks in the complement $\mathbb{B}^4\backslash L$. These cycles are crucial ingredients in describing cluster structures coming from sheaf-theoretic or Floer-theoretic invariants of Legendrian links. From a cluster-theoretic perspective, the class of 3-braid closures is of particular interest as it contains the smallest known examples of Legendrians with augmentation varieties  with multiple irreducible components; see \cite[Section 2.2.2]{LipshitzNg} for an explicit computation. We hope to further explore the construction and classification of fillings of the orientably fillable Legendrian links identified here in future work.

Our next result concerns non-orientable fillability of quasipositive $3$-braids.

\begin{theorem}\label{thm: intro_non-orientable_fillability}
   The closure $\hat{\beta}$ of a quasipositive $3$-braid $\beta$ is non-orientably exact Lagrangian fillable if $m\leq -2$. If $m\geq -1$, then $\hat{\beta}$ does not admit any decomposable non-orientable exact Lagrangian fillings.
\end{theorem}

See Section~\ref{ssec:background-fillings} for the definition of decomposable. The case when $m \ge 0$ is proven in \cite{CCPRSY} using a normal ruling obstruction. We apply the same technique to the case $m = -1$. For $m \leq -2$, we construct a decomposable non-orientable exact Lagrangian filling of the max-tb Legendrian representative $\La(\beta)$ based on a normal ruling. In contrast with Theorem ~\ref{thm: fillability},  Theorem~\ref{thm: intro_non-orientable_fillability} does not cover all non-orientably fillable Legendrians of braid index 3 due to the fact that, as mentioned above, our methods rely on the quasipositivity characterization of \cite{Orevkov04}. In particular, there are known examples, such as the figure-eight knot, of 3-braid closures that are not quasipositive and are non-orientably fillable. Note that the case of $m=-2$ provides a new family of examples of Legendrian links that admit both orientable and non-orientable fillings, behavior originally observed in \cite{CCPRSY}.

\begin{remark}
While we only construct a single non-orientable filling for each Legendrian in our proof of Theorem~\ref{thm: intro_non-orientable_fillability}, we expect that a similar approach could produce a collection of (possibly distinct) fillings. One could reasonably hope to apply cluster theoretic methods similar to those used in \cite{CSHW2023} to distinguish them by their induced (ungraded) augmentations of the Legendrian contact dg-algebra. 
\end{remark}

The remainder of this paper is organized as follows:  We review background on quasipositive $3$-braids, on Legendrian links and their ruling invariants, and on constructions and obstructions of exact Lagrangian fillings of Legendrian links in Section~\ref{sec:background}. In Section~\ref{sec:max-tb}, we describe the max-tb Legendrian representatives of quasipositive $3$-braid closures touched on above and associate with them a ruling. We prove Theorem~\ref{thm: fillability} in Section~\ref{sec: orientable-fill} and Theorem~\ref{thm: intro_non-orientable_fillability} in Section~\ref{sec:non-orientable-fill}. 

% **********
\subsection*{Acknowledgments}
Thank you to Lenny Ng for helpful advice and feedback throughout this project. We also thank Josh Sabloff for the question that led us to consider non-orientable fillability and for discussing his conjecture with us. Zijun Li was involved in an earlier version of this project and we are grateful for his contribution. Finally, thanks to the anonymous referee for helpful feedback and suggestions.

\section{Background}
\label{sec:background}
\subsection{Quasipositive $3$-braids}
\label{ssec:3-braids}
Let $\Br_3$ denote the $3$-stranded braid group generated by $\sigma_1,\sigma_2$, where $\sigma_i$ corresponds to a single positive crossing between strands $i$ and $i+1$ of the braid, subject to the Artin relation $\sigma_1 \sigma_2\sigma_1 = \sigma_2\sigma_1\sigma_2$. Denote by $\Br_3^+$ the submonoid of $\Br_3$ generated by $\sigma_1$ and $\sigma_2$. Also, let $\Delta=\sigma_1\sigma_2\sigma_1$ denote the half-twist on $3$ strands, i.e. the \dfn{Garside element} of $\Br_3$. Then, any element $\beta\in \Br_3$ can be written in the \dfn{Garside normal form} 
\begin{equation*}
    \beta=\beta_+\Delta^m,
\end{equation*}
for some $m\in \Z$ and $\beta_+\in \Br_3^+$ not containing any power of $\Delta$. This presentation is unique up to the Artin relation by \cite[Corollary 2.3]{Orevkov04}.

We say that a braid $\beta$ is \dfn{positive} if it is in $\Br_3^+$ and \dfn{quasipositive} if $\beta = \prod_j a_jx_ja_j^{-1}$ for $a_j \in \Br_3$ and $x_j \in \Br_3^+$. Observe that a positive braid is necessarily quasipositive. These notions of positivity can be equivalently stated using the Garside normal form. First, note that $\beta$ is positive if and only if $m \ge 0$. For $m < 0$, the following criterion of quasipositivity is due to Orevkov.

\begin{theorem}[\cite{Orevkov04}, Proposition 3.1]
\label{thm:quasipositive-criterion}
    Let $\beta = \beta_+\Delta^m$ where $\beta_+ \in \Br_3^+$ and $m \leq 0$. The braid $\beta$ is quasipositive if and only if one can delete letters from $\beta_+$ so that the obtained word is equal in $\Br_3^+$ to the word $\Delta^{-m}$.
\end{theorem}

We will consider the ($0-$framed) closure $\hat{\beta}$ of a braid $\beta \in \Br_3$, whose smooth isotopy class is characterized by $\beta$ up to not only the Artin relations but also conjugation. In this setting, we can conjugate the Garside normal form of $\beta$ and write more concretely
\begin{equation*}
    \beta=\sigma_1^{k_1}\sigma_2^{k_2}\dots \sigma_1^{k_{r-1}}\sigma_2^{k_r}\Delta^m,
\end{equation*}
where $k_r=0$ if and only if $m$ is odd, and otherwise $k_i\geq 2$ for all $i$. Using the uniqueness of Garside normal form in $\Br_3$, it is not hard to see that $m$ is uniquely determined to be greatest among all braid words realizing the knot type $\hat{\beta}$. In particular, conjugating $\beta$ in the form of Equation~\ref{eq:garside-normal} does not increase $m$. 

A smooth link is \dfn{positive} if it has a projection in which all crossings are positive. However, not all positive links are closures of positive braids; see, for example, \cite{Hedden??} and references therein. On the other hand, a link is \dfn{quasipositive} if it is the closure of a quasipositive braid. Note that there is no ambiguity in this definition: if $\beta = \prod_j a_jx_ja_j^{-1}$ is a quasipositive braid, observe that the conjugate $\alpha\beta\alpha^{-1}$ is still quasipositive by inserting $\alpha\alpha^{-1}$ or $\alpha^{-1}\alpha$ in between the products $a_jx_ja_j^{-1}$. Thus, $\hat{\beta}$ is quasipositive as a smooth link if and only if $\beta$ is quasipositive as a braid. In terms of Equation~\ref{eq:garside-normal}, Orevkov's characterization of quasipositive $3$-braids implies that $m \ge 0$ or $r \ge - \lfloor \frac{3m-2}{2} \rfloor$ when $m < 0$. We note that this lower bound on $r$ is not sharp but suffices for the purpose of this paper.

For the rest of the paper, when we use $\hat{\beta}$ to denote the smooth link given as the closure of a quasipositive $3$-braid $\beta$, we assume that $\beta$ is in the specific Garside normal form given by Equation~\ref{eq:garside-normal}. 

\subsection{Legendrian links in $\R^3$}\label{ssec:links_background}

We briefly review the basic geometric terminology for Legendrian links in $\R^3$ that we will need for this paper. We refer readers to \cite{Etnyre:intro} for an in-depth introduction to these materials. The standard contact structure $\xi_{st}$ in $\R^3$ is the 2-plane field given by the kernel of the 1-form $\alpha_{\st}=dz-ydx$. A link $\La \subseteq (\R^3, \xi_{st})$ is \dfn{Legendrian} if it is everywhere tangent to $\xi_{\text{st}}$. In this paper, all Legendrian links are oriented and we consider Legendrian links up to Legendrian isotopy, i.e. ambient isotopy through a family of Legendrians.

We will often picture a Legendrian link via its front projection, namely the image of the link under the projection $\Pi_{xz}:\R^3 \to \R^2$ to the $xz$ plane. There are two classical invariants associated to a Legendrian link. The \dfn{Thurston-Bennequin number} $\tb(\Lambda)$ is the writhe of $\Pi_{xz}(\Lambda)$ minus the number of left (or right) cusps and the \dfn{rotation number} $\rot(\Lambda)$ is given by $\frac{1}{2}(c_- - c_+)$ where $c_-,c_+$ are the number of cusps oriented downwards and upwards, respectively.  

Any smooth link $K$ has a Legendrian $C^0$ approximation. It follows from the Bennequin inequality that the Thurston-Bennequin number of a Legendrian representative of $K$ is bounded above and we denote by $\overline{\tb}(K)$ the maximum tb attained over all Legendrians of smooth link type $K$. Alternatively, the HOMFLY polynomial $P_K(v,z)$ also gives a bound on the classical invariants. In particular, for any Legendrian representative $\Lambda$ of $K$, Fuchs and Tabachnikov showed in \cite{FuchsTabachnikov97} that
\begin{equation*}
    \tb(\Lambda) + |\rot(\Lambda)| \leq -d_{P_K},
\end{equation*}
where $d_{P_K}$ is the lowest degree in $v$ of $P_K(v,z)$. Thus, $\overline{\tb}(K) \leq -d_{P_K}$ and we refer this as the \dfn{HOMFLY bound} on the maximum Thurston-Bennequin number of $K$.

\begin{center}
    \begin{figure}[h!]{ \includegraphics[width=.3\textwidth]{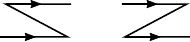}}\caption{Local picture of stabilization.}
    \label{fig: stabilization}\end{figure}
\end{center}

Given a Legendrian link $\La$, there is a local move that produces another Legendrian link in the same smooth knot type called stabilization and we describe it in the front diagram. First, the \dfn{stabilization} of $\La$ is obtained by removing a strand of $\La$ and replacing it with one of the zig-zags shown in Figure~\ref{fig: stabilization}. Assuming $\La$ is oriented as shown, the left zig-zag gives a positive stabilization and the right one produces a negative stabilization, denoted by $S_+(\La)$ and $S_-(\La)$, respectively. One may check that $\tb(S_\pm(\La)) = \tb(\La) - 1$ and $\rot(S_\pm(\La)) = \rot(\La) \pm 1$. The \dfn{destabilization} of $\La$ is the reverse process of stabilization and may not always be possible for an arbitrary Legendrian.

Of particular interest in this paper are Legendrians obtained from the closure $\hat{\beta}$ of a $3$-braid $\beta$, which we now describe using their front projections. Suppose $\hat{\beta}$ is presented by drawing $\beta$ horizontally and joining the left and right ends of $\beta$ by a nested set of non-intersecting arcs. We fix an orientation that goes counterclockwise around such a closure. If $\beta$ is positive, a Legendrian representative of $\hat{\beta}$ is obtained by simply replacing all vertical tangencies with cusps; namely, the \dfn{rainbow closure} of $\beta$. See Figure \ref{fig: rainbow} (left). \cite[Theorem 3.4]{Satellite1} proves that rainbow closures are the unique max-tb Legendrian representatives of braid positive links. Thus, observe from Figure~\ref{fig: rainbow} (left) that $\overline{\tb}(\hat{\beta}) = |\beta| - 3$, where $|\beta|$ is the signed word length of $\beta$. 

If in its Garside normal form, $\beta = \beta_+\Delta^{-1}$, we isotope $\hat{\beta}$ to the $(-1)$-closure of $\beta_+\Delta$ and replace all vertical tangencies with cusps; see Figure~\ref{fig: rainbow} (right). \cite{CGGLSS} shows that any Legendrian of this form is orientably fillable and is thus a max-tb Legendrian representative by Lemma~\ref{lem:orientable-obs} and Corollary~\ref{cor: Ruling_TB}. Observe from Figure~\ref{fig: rainbow} (right) that $\overline{\tb}(\hat{\beta}) = |\beta|$.

Lastly, if $m < -1$, we isotope $\hat{\beta}$ to arrange one negative half twist vertically to the right of $\beta_+$ and the remaining $-m-1$ negative half twists vertically to the left of $\beta_+$. A replacement of all vertical tangencies with cusps yields a Legendrian representative of $\hat{\beta}$. We denote this Legendrian by $\La_0(\beta)$ if $m$ is even; see Figure~\ref{fig: initial_fronts} (left). Otherwise, we see a stabilization at the top of the front where we can immediately destabilize once to produce $\La_{0}(\beta)$, pictured in Figure~\ref{fig: initial_fronts} (right). Note that $\tb(\Lambda_0(\beta)) = |\beta| + \lfloor \frac{3m}{2} \rfloor$, and  $\rot(\Lambda_0(\beta)) = -\lfloor \frac{3(m+2)}{2} \rfloor$. In contrast with the cases $m \ge -1$ above, $\Lambda_0(\beta)$ is often not a max-tb representative. However, it will serve as an initial front that we iteratively destabilize in order to produce a (possibly non-unique) max-tb representative in Section~\ref{sec:max-tb}.

\begin{center}
    \begin{figure}[h!]{ \includegraphics[width=.8\textwidth]{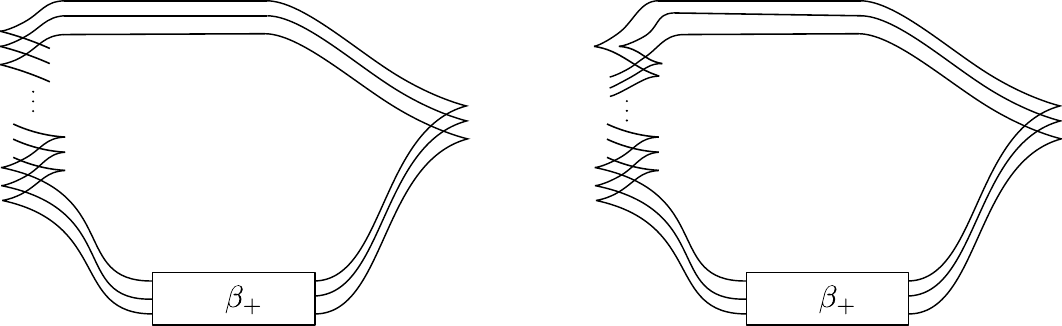}}\caption{Initial fronts $\Lambda_0(\beta)$ for $m$ even (left) and odd (right).}
    \label{fig: initial_fronts}\end{figure}
\end{center}

\subsection{Rulings of Legendrian links}

Normal rulings are combinatorial Legendrian invariants developed independently by Fuchs \cite{Fuchs03} and Chekanov--Pushkar \cite{ChekanovPushkar05}. They are intimately tied to augmentations of the Legendrian contact dg-algebra \cite{Fuchs03, FuchsIshkhanov, Sabloff05,  HenryRutherford15}, as well as the existence of exact Lagrangian fillings \cite{atiponrat2015, CCPRSY}. Here we define rulings and discuss some of their properties most relevant to fillability.

    Given a Legendrian $\La$ with a fixed front projection $D$, a \dfn{normal ruling} $\rho(D)$ of $\La$ is a set of crossings $s(\rho)$ called \dfn{switches} such that resolving each of the switches yields a union of max-tb unknots $\La_1, \dots, \La_m$ satisfying
    \begin{itemize}
        \item each $\La_i$ has  crossings, two cusps, and bounds an embedded disk
        \item exactly two link components are incident to any switches
        \item in a neighborhood of each switch, the incident ruling disks $D_i$ are either completely nested or disjoint, as in the top row of Figure \ref{fig: RulingDef}
    \end{itemize}

  A crossing with no switches is referred to as either a \dfn{departure} or \dfn{return} if it resembles the middle or bottom row, respectively, of Figure \ref{fig: RulingDef}. We denote by $d(\rho)$ (resp. $r(\rho)$) the set of departures (resp. returns) of $\rho$.  

\begin{center}
    \begin{figure}[h!]{ \includegraphics[width=.8\textwidth]{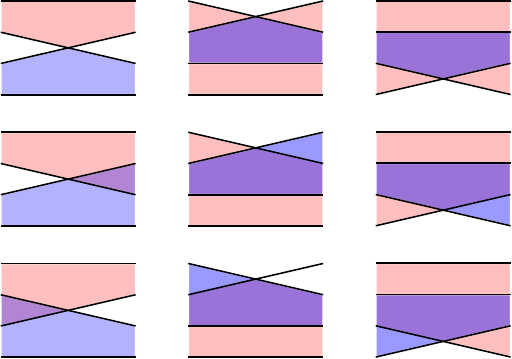}}\caption{All possible configurations of ruling disks in a neighborhood of switches (top), departures (middle), and returns (bottom).}
    \label{fig: RulingDef}\end{figure}
\end{center}

The set of rulings of $D$ is a Legendrian isotopy invariant of $\La$ \cite{ChekanovPushkar05}, hence we may discuss the rulings of $\La$ rather than the rulings of the front diagram $D$. 

A normal ruling $\rho$ is \dfn{oriented} if every switch of $\rho$ occurs at a positive crossing, and unoriented otherwise. We define the \dfn{ruling polynomial} of $\La$ to be
$$R_\La(z)=\sum_{\text{rulings } \rho} z^{s(\rho)-r(\rho)+1}.$$ 

We define the \dfn{oriented ruling polynomial} $R_\La^\circ(z)$ of $\La$ analogously by summing over all oriented rulings. 

Given a smooth knot type $K$, normal rulings of its Legendrian representatives are closely related to various bounds on the maximum Thurston-Bennequin number of $K$.     
\begin{theorem}[\cite{Rutherford06}] 
Given a Legendrian link $\La$ of smooth knot type $K$,
\begin{enumerate}
    \item the ruling polynomial $R_\La(z)$ is the coefficient of $a^{-tb(\La)-1}$ in the Kauffman polynomial $F_K(a, z)$.
    \item the oriented ruling polynomial $R^\circ_\La(z)$ is the coefficient of $a^{-tb(\La)-1}$ in the HOMFLY polynomial $P_K(a, z)$.
\end{enumerate}
\end{theorem}

Upper bounds on $tb(\La)$ coming from the HOMFLY and Kaufmann polynomials then yield the following corollary.

\begin{corollary}[\cite{Rutherford06}] \label{cor: Ruling_TB}
A Legendrian link $\La$ admits a(n oriented)  normal ruling if and only if the Kaufmann bound (resp. HOMFLY bound) on $\tb(\La)$ is sharp. Moreover, if $\La$ admits a normal ruling, it maximizes tb in its smooth isotopy class.
\end{corollary}

\subsection{Exact Lagrangian fillability}\label{ssec:background-fillings}

We end this section by a recollection of some results on exact Lagrangian cobordisms and fillings of a Legendrian link that we use in this paper. First, the symplectization $\Symp(M, \ker(\alpha))$ of a contact manifold $(M, \ker(\alpha))$ is the symplectic manifold $(\R_t\times M, d(e^t\alpha))$. Given two Legendrian links $\La_-, \La_+\subseteq (\R^3,\xi_{\st})$, an \dfn{exact Lagrangian cobordism} $L\subseteq \Symp(\R^3, \ker(\alpha_{\st}))$ from $\La_-$ to $\La_+$ is a cobordism $\Sigma$ such that there exists some $T>0$ satisfying the following: 

	\begin{enumerate}
		\item $d(e^t\alpha_{\st})|_\Sigma=0$ %Lagrangian
		\item $\Sigma\cap ((-\infty, T]\times \R^3)=(-\infty, T]\times \La_-$ %Cylindrical ends
		\item $\Sigma\cap ([T, \infty)\times \R^3)=[T, \infty) \times \La_+$ %Cylindrical ends
		\item $e^t\alpha_{\st}|_\Sigma=df$ for some function $f: \Sigma\to \R$ that is constant on $(-\infty, T]\times \La_-$ and $[T, \infty)\times \La_+$. %Exact
	\end{enumerate}

An exact Lagrangian filling of the Legendrian link $\La\subseteq (\R^3, \xi_{\st})$ is an exact Lagrangian cobordism $L$ from $\emptyset$ to $\La$ that is embedded in $\Symp(\R^3, \ker(\alpha_{\st}))$. In this paper, we say that a Legendrian link is \dfn{orientably fillable} if it admits an orientable exact Lagrangian filling and a smooth knot type $K$ is \dfn{orientably fillable} if it has an orientably fillable Legendrian representative. Similar terminology is defined for \dfn{non-orientable fillability}.

In the front projection, one can construct exact Lagrangian cobordisms via traces of Legendrian isotopies and the two additional elementary cobordisms given in Figure \ref{fig: minima_saddle}. Cobordisms constructed using these pieces are referred to as \dfn{decomposable}.  We will also rely on a particular Lagrangian 1-handle attachment built out of several elementary exact Lagrangian cobordisms, as pictured in Figure \ref{fig: D4-}. We refer to this cobordism as a \dfn{pinching cobordism}, as it is Hamiltonian isotopic, relative to the boundary, to the crossing resolution cobordism in the Lagrangian projection  \cite[Proposition 3.1]{Hughes2021b}.\footnote{This cobordism also appears in the literature as Lagrangian projection of the Legendrian lift of a generic perturbation of the $D_4^-$ singularity in a Legendrian surface wavefront; See \cite[Section 2.1.2]{Hughes2021b} for more details.}

\begin{center}
    \begin{figure}[h!]{ \includegraphics[width=.8\textwidth]{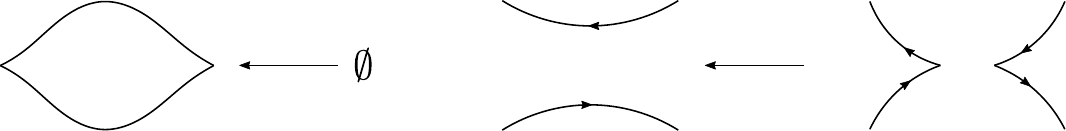}}\caption{Minimum and (orientable) saddle cobordisms depicted in the front projection. The horizontal arrows point in the direction of expanding symplectic area.}
    \label{fig: minima_saddle}\end{figure}
\end{center}

\begin{center}
    \begin{figure}[h!]{ \includegraphics[width=.8\textwidth]{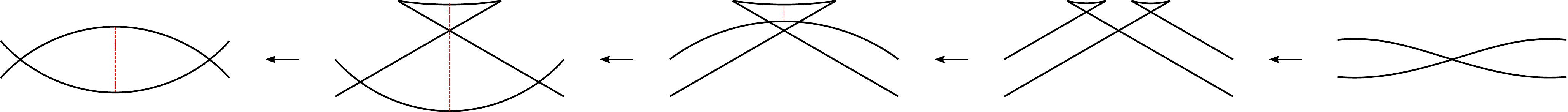}}\caption{Pinching cobordism depicted in the front projection. The dashed red line depicts a contractible Reeb chord. The arrows point in the direction of expanding symplectic area.}
    \label{fig: D4-}\end{figure}
\end{center}

Consider now a smooth knot type $K$ obtained from the closure of a braid $\beta$. By constructing a filling explicitly from an oriented normal ruling in which all crossings are switched, Hayden-Sabloff proves the following result.

\begin{theorem}
[\cite{HaydenSabloff}, Theorem 1.1]\label{thm:haydensabloff}
    If $\beta$ is positive, then $K$ is orientably fillable. 
\end{theorem}

On the obstruction side, it is known that quasipositivity of $\beta$ is a necessary condition for $K$ to be fillable: a filling of $K$ can be perturbed to a symplectic filling of a transverse knot \cite{MR1362834} and a smooth knot type with a symplectically fillable transverse representative is quasipositive \cite{BoileauOrevkov01}.

Another notable obstruction of orientable fillability is the sharpness of HOMFLY bound on $\overline{\tb}$. In particular, an orientable filling induces an augmentation of the Legendrian contact homology dg-algebra of $\Lambda$, which further produces an orientable ruling under the correspondence in \cite{Fuchs03, FuchsIshkhanov, Sabloff05, Leverson}. The claimed obstruction is then a consequence of Corollary \ref{cor: Ruling_TB}. In summary,

\begin{lemma} [\cite{HaydenSabloff}]\label{lem:orientable-obs}
    If a Legendrian $\Lambda$ is orientably fillable, then it is quasipositive and the HOMFLY bound on $\overline{\tb}(K)$ is sharp.
\end{lemma}

It is a conjecture of Hayden-Sabloff that these two necessary conditions together also give the sufficient condition for the orientable fillability of $K$. We shall consider their conjecture in the case when $K$ is the closure of a $3$-braid. In this context, Theorem~\ref{thm:haydensabloff} deals with the case when the integer $m$ from Equation \ref{eq:garside-normal} is non-negative. When $m = -1$, the fact that $K$ is orientably fillable follows from a general framework developed in \cite{CGGLSS}. In more detail, for $\beta \in \Br_n^+$, the authors construct a graph encoding an (orientable) exact Lagrangian filling of the $(-1)$-closure of $\beta\Delta$, smoothly the braid closure of $\beta\Delta^{-1}.$ One can produce a filling from this graph by applying the Legendrian weave construction of \cite[Section 2]{CasalsZaslow} to the graph after performing the minor modification  discussed in \cite[Definition 7.2]{CasalsNho25}.

We shall also consider non-orientable fillability of $K$. The case when $K$ is positive is addressed in \cite{CCPRSY}.

\begin{theorem}[\cite{CCPRSY}, Theorem 1.3] \label{thm:ccpr+}
    If $K$ is positive, it is not non-orientably decomposably fillable.
\end{theorem}

The proof of this theorem makes use of the following ruling obstruction of non-orientable decomposable fillings, which we will also exploit in Section~\ref{sec:non-orientable-fill}.

\begin{lemma}[\cite{CCPRSY}, Corollary 3.8] \label{lem: non-orientable obstruction}
Let $\La$ be a Legendrian link with underlying smooth isotopy class $K$.    If $R_\La(z)=R^\circ_\La(z)$, then no Legendrian representative of $K$ admits a decomposable non-orientable fillings.
\end{lemma}

\section{Max-tb representatives of Quasipositive $3$-braid Closures}\label{sec:max-tb}
Before proving our main theorems, we produce a max-tb Legendrian representative $\La_k(\beta)$ of $\hat{\beta}$ when $m \leq -2$. We will use these Legendrian representatives for obstructing orientable exact Lagrangian fillings of $\hat{\beta}$ when $m < -2$ and constructing non-orientable fillings when $m \leq -2$. Combined with previous results in the case $m \ge -1$ (c.f. Section~\ref{ssec:links_background}), we shall prove Corollary~\ref{cor:max-tb} regarding the maximum Thurston-Bennequin number of quasipositive $3$-braid links.
\subsection{Constructing $\La_k(\beta)$}
Beginning with the Legendrian representative $\La_0(\beta)$, introduced in Section~\ref{ssec:links_background}, we give an iterative destabilization process to produce a sequence $\La_0(\beta), \La_1(\beta), \dots, \La_k(\beta)$ that terminates in a max-tb representative $\La_k(\beta)$ of $\hat{\beta}$ for $k=-\lfloor \frac{m+2}{2}\rfloor$. The key ingredient in this process is a local smooth isotopy that we call a ZZS destabilization drawn in the bottom row of Figure~\ref{fig: DestabilizationMove}. It involves an SZ move and a negative destabilization, which increases the tb of the link by $1$ and decreases its rotation number by $3$. In figures below, we indicate applications of ZZS destabilization by a red box.

\begin{center}
    \begin{figure}[h!]{ \includegraphics[width=.8\textwidth]{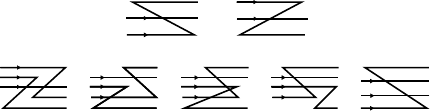}}\caption{(Top) An SZ move. (Bottom) A ZZS destabilization, where the first step is an SZ move and the last step is a negative destabilization.}
    \label{fig: DestabilizationMove}\end{figure}
\end{center}

Recall the notation $\beta=\sigma_1^{k_1}\sigma_2^{k_2}\dots \sigma_1^{k_{r-1}}\sigma_2^{k_r}\Delta^m$ from Equation \ref{eq:garside-normal}.
If $m = -2$, we maintain $\La_0(\beta)$ with further modifications to come below. Otherwise, since $k_1\geq 2$, we begin by rewriting $\beta_+$ as $\sigma_1^2\beta_{1}$ for $\beta_1=\sigma_1^{k_1-2}\sigma_2^{k_2}\dots \sigma_1^{k_{r-1}}\sigma_2^{k_r}\in \Br_3^+$. We can then perform the sequence of isotopies and the ZZS destabilization pictured in the top row in Figure \ref{fig: braid_destab} followed by the isotopy pictured in Figure~\ref{fig: zigzag_rot} (left) to produce $\La_1(\beta)$. Note that in terms of the braid word, the isotopy in Figure~\ref{fig: braid_destab} is equivalent to replacing $\Delta^{-2}\sigma_1^2$ by $\sigma_2^{-1}\sigma_1^{-2}\sigma_2^{-1}$. If $m=-3$ or $-4$, we keep $\La_1(\beta)$ with further modification to follow below. Otherwise, we continue destabilizing $\La_i(\beta)$ in this manner, considering the following four cases at each step:
\begin{enumerate}
    \item $\beta_i=\sigma_1^{2}\beta_{i+1}$ 
    \item $\beta_i=\sigma_2\sigma_1^2\beta_{i+1}$ 
    \item $\beta_i=\sigma_2^2\beta_{i+1}$ 
    \item $\beta_i=\sigma_1\sigma_{2}^2\beta_{i+1}$ 
\end{enumerate}
where $\beta_i\in \Br_3^+$ for all $1\leq i \leq k$. Since $k_i\geq 2$ for $1\leq i \leq r-1$, these cases are exhaustive.
 
\begin{center}
    \begin{figure}[h!]{ \includegraphics[width=.9\textwidth]{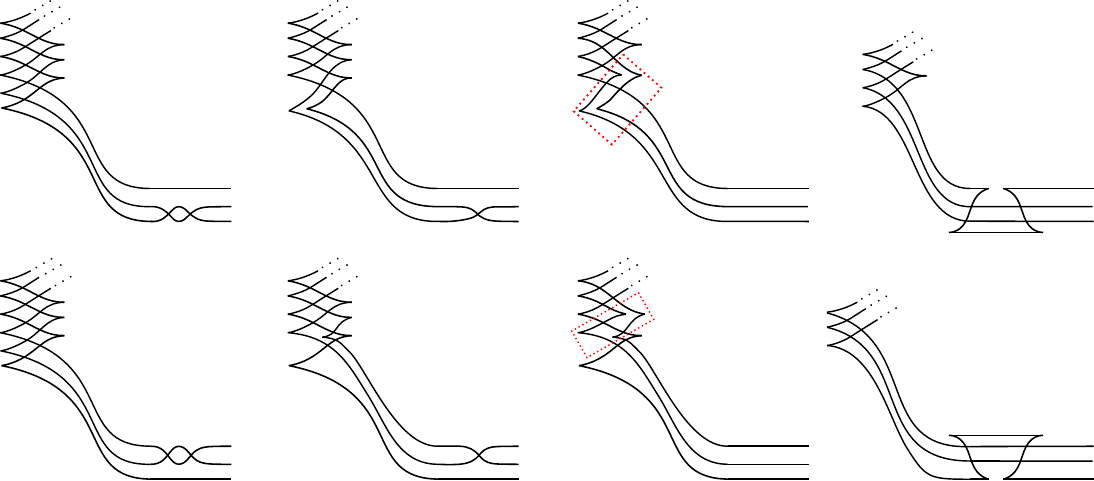}}\caption{Destabilization of $\La_i(\beta), 1\leq i\leq k-1$  when $\beta_i$ begins with either $\sigma_1^2$ or $\sigma_2\sigma_1^2$ (top) or either $\sigma_2^2$ or $\sigma_1\sigma_2^2$ (bottom).}
    \label{fig: braid_destab}\end{figure}
\end{center}

\begin{center}
    \begin{figure}[h!]{ \includegraphics[width=.91\textwidth]{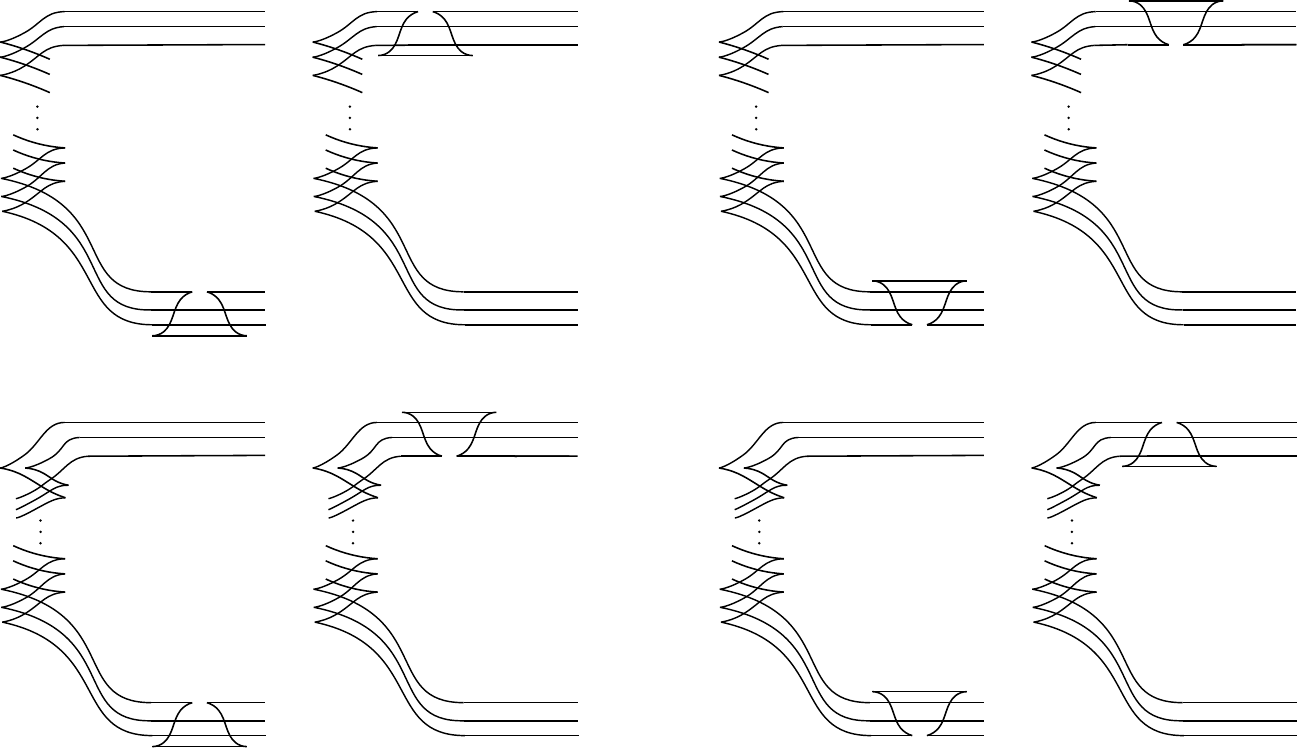}}\caption{Legendrian isotopies performed after applying destabilization given in Figure~\ref{fig: braid_destab} in Cases (1) (left) and (3) (right) when $m$ is even (top) or $m$ is odd (bottom).}
    \label{fig: zigzag_rot}\end{figure}
\end{center}

\begin{center}
    \begin{figure}[h!]{ \includegraphics[width=.9\textwidth]{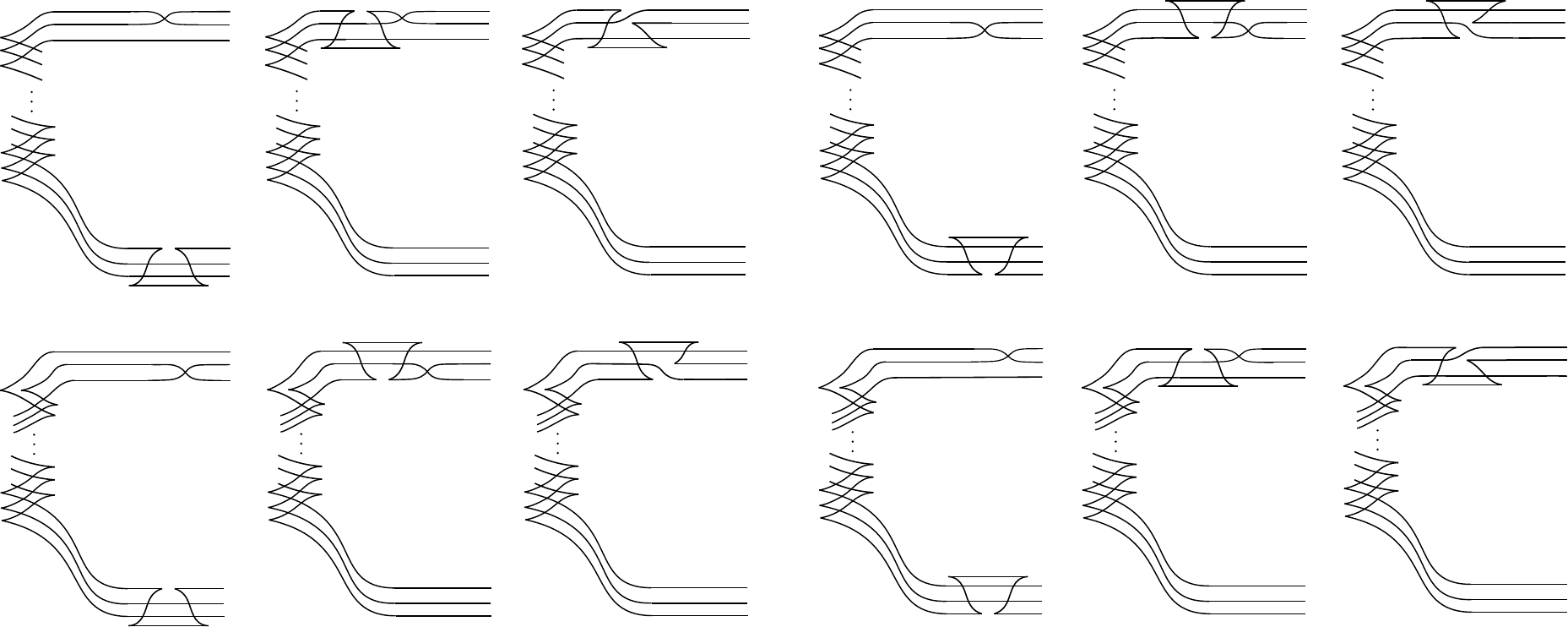}}\caption{Legendrian isotopies performed after applying destabilization given in Figure~\ref{fig: braid_destab} in Cases (2) (left) and (4) (right) when $m$ is even (top) or $m$ is odd (bottom).}
    \label{fig: oddk_i_rotation}\end{figure}
\end{center}
  
 In Case (1), we repeat the same destabilization procedure depicted in the top row of Figure \ref{fig: braid_destab} followed again by the isotopy from Figure \ref{fig: zigzag_rot} (left) to produce $\La_{i+1}(\beta)$. In Case (2), we rotate the crossing $\sigma_2$ to the top of the front diagram, then perform the destabilization procedure depicted in the top row of Figure \ref{fig: braid_destab}, followed by the isotopy from Figure~\ref{fig: oddk_i_rotation} (left) 
 to produce $\La_{i+1}(\beta)$. Similarly, in Case (3), we perform the destabilization procedure depicted in the bottom row of Figure \ref{fig: braid_destab} followed by the isotopy from Figure \ref{fig: zigzag_rot} (right) to produce $\La_{i+1}(\beta)$. In Case (4), we rotate the crossing $\sigma_1$ to the top of the front diagram, then perform the destabilization procedure depicted in the bottom row of Figure \ref{fig: braid_destab}, followed by the isotopy from Figure~\ref{fig: oddk_i_rotation} to produce $\La_{i+1}(\beta)$. 
  
  If $m$ is even, including $m = -2$, we can continue this process until we only have one negative half twist remaining on the left, resulting in $\Lambda_k(\beta)$. In particular, note that each destabilization reduces the number of negative half twists on the left by two. One can readily prove that $\Lambda_k(\beta)$ is a max-tb representative, but for the proof of Proposition \ref{prop: non-orientably fillable} below, we perform an extra Reidemeister III move on the right closure of $\Lambda_k(\beta)$ as depicted in Figure~\ref{fig:  R3_final} (left). This is made possible by the fact that $k_r \ge 2$ in this case. If $\beta_{k-1}$ begins with $\sigma_1^2$ or $\sigma_2\sigma_1^2$. we perform another Reidemeister III move on the left closure of $\Lambda_k(\beta)$ as depicted in Figure~\ref{fig:  R3_final} (center, right). We continue to denote the end result by $\Lambda_k(\beta)$.

    \begin{center}
    \begin{figure}[h!]{ \includegraphics[width=\textwidth]{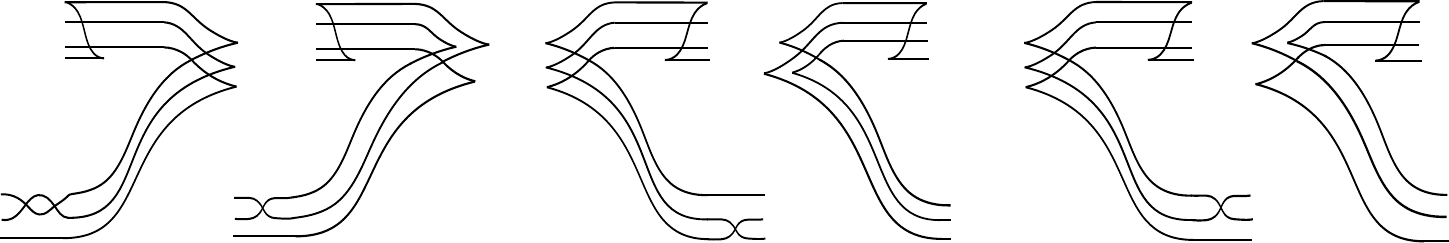}}\caption{Reidemeister III moves performed for the case when $m$ is even. We perform these moves on the left closure only when $\beta_{k-1}$ begins with either $\sigma_1^3$ or $\sigma_2\sigma_1^3$ (center) or  $\sigma_1^2\sigma_2$ or $\sigma_2\sigma_1^2\sigma_2$ (right).}
    \label{fig: R3_final}\end{figure}
    \end{center}
  
  If $m$ is odd, we iterate the process above to obtain $\La_{k-1}(\beta)$ and treat the last iteration differently as follows. If $\La_{k-1}(\beta)$ is of the form described in Case (1) or (2), we apply the destabilization procedure depicted in the top row of Figure~\ref{fig: oddm_destab}. In Case (2), this is done after rotating the crossing $\sigma_1$ and we perform an additional Reidemeeister II move depicted in Figure \ref{fig: final_oddk_i} (left) in the end. If $\La_{k-1}(\beta)$ is of the form described in Case (3) or (4), we apply the destabilization procedure depicted in the bottom row of Figure~\ref{fig: oddm_destab}. Again, in Case (4), this is done after rotating the crossing $\sigma_1$ and we perform an additional Reidemeister II move depicted in Figure \ref{fig: final_oddk_i} (right) in the end.

  \begin{center}
    \begin{figure}[h!]{ \includegraphics[width=.9\textwidth]{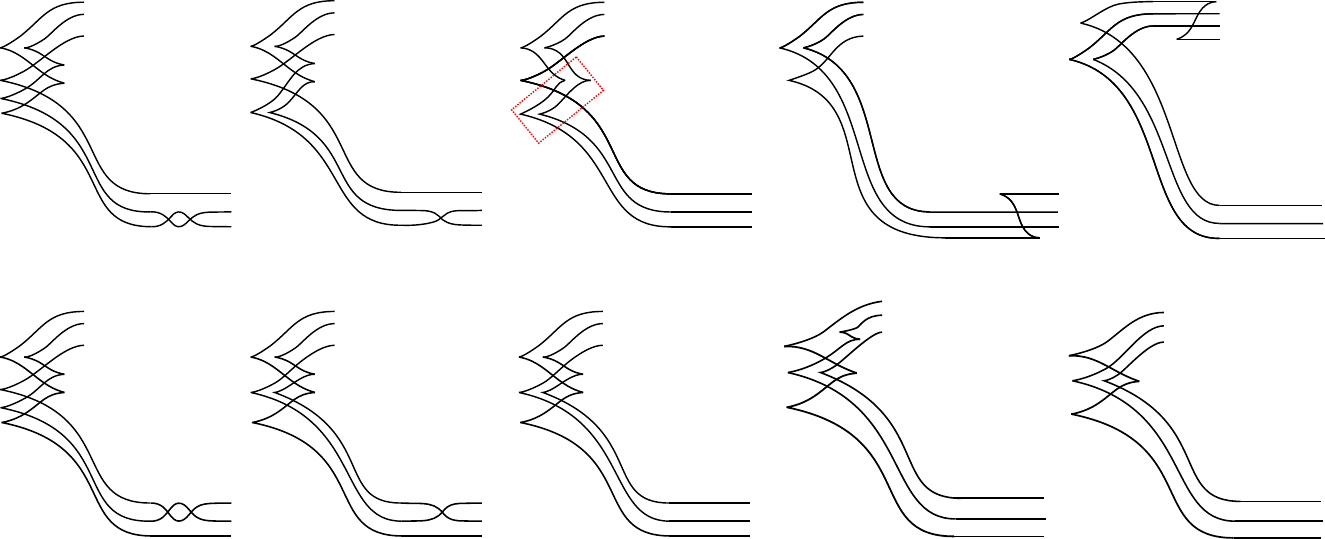}}\caption{Destabilization of $\La_{k-1}(\beta)$ when $m$ is odd and $\beta_{k-1}$ begins with $\sigma_1^2$ or $\sigma_2\sigma_1^2$ (top) and $\sigma_2^2$ or $\sigma_1\sigma_2^2$ (bottom).}
    \label{fig: oddm_destab}\end{figure}
\end{center}

\begin{center}
    \begin{figure}[h!]{ \includegraphics[width=.75\textwidth]{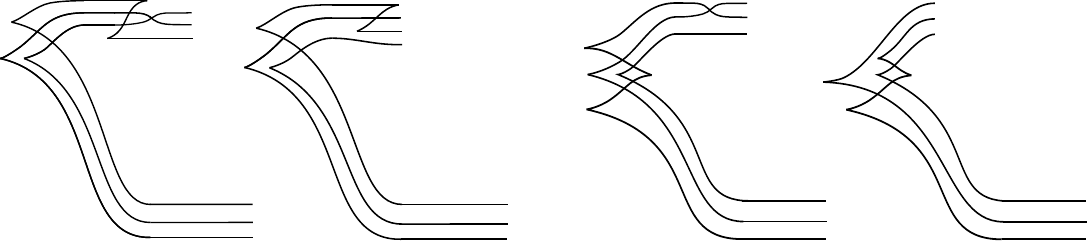}}\caption{Reidemeister II moves performed for case when $m$ is odd and $\beta_{k-1}$ begins with a $\sigma_2\sigma_1^2$ (left) or $\sigma_1\sigma_2^2$ (right).}
    \label{fig: final_oddk_i}\end{figure}
\end{center}
  
  At each step, since $\beta$ is quasipositive, there are at least $|\Delta^{m-2i}|=3(m-2i)$ crossings so that $\beta_i$ is always of the form specified by one of Cases (1)-(4). The end result of this procedure is the Legendrian $\La_k(\beta)$ with front projection given by the positive braid $\beta_k$ and the pieces of the form depicted in Figures~\ref{fig: closure_rulings-2}, ~\ref{fig: closure_rulings_even}, or ~\ref{fig: closure_rulings_odd}, and \ref{fig: zigzag_rulings}: the left and right closures, as well as a sequence of zigzags on the top. See Figure \ref{fig: example_front_ruling} for an example.

   \begin{center}
    \begin{figure}[h!]{ \includegraphics[width=.7\textwidth]{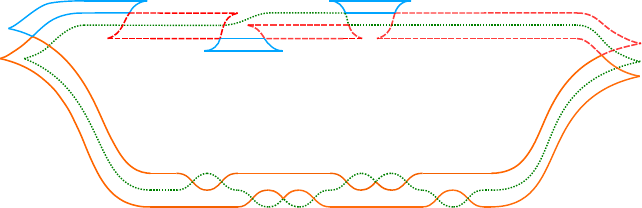}}\caption{A ruling associated to $\La_3(\sigma_1^3\sigma_2^2\sigma_1^2\sigma_2^2\sigma_1^3\sigma_2^3\sigma_1^2\Delta^{-7})$.}
    \label{fig: example_front_ruling}\end{figure}
    \end{center}

Note that, aside from the final destabilization when $m$ is odd and $\beta_{k-1}$ begins with $\sigma_2^2$ or $\sigma_1\sigma_2^2$, each destabilization is a ZZS move and thus decreases rotation number by 3. In consequence, we can compute thet rotation number for our Legendrian representatives as 
$$\rot(\La_k(\beta))= - \left\lfloor \frac{3(m+2)}{2} \right \rfloor+3\left \lfloor \frac{m+4}{2} \right \rfloor-1$$ when $m$ is odd and $\beta_{k-1}=\sigma_2^2\beta_k$ or $\beta_{k-1}=\sigma_1\sigma_2^2\beta_k$, and $$\rot(\La_k(\beta))=- \left \lfloor \frac{3(m+2)}{2} \right \rfloor+3 \left \lfloor \frac{m+2}{2} \right \rfloor$$ 
otherwise. This yields the following:
\begin{align*}
    \tb(\La_k(\beta)) &= %\left \lfloor \frac{3m}{2} \right \rfloor- \left \lfloor \frac{m+2}{2} \right \rfloor=  
    |\beta| + m - 1,\\ %\left \lfloor \frac{3m}{2} \right \rfloor- \left \lfloor \frac{m+2}{2} \right \rfloor, \\
    \rot(\La_k(\beta)) &= \begin{cases}
       % - \lfloor \frac{3(m+2)}{2} \rfloor+3\left \lfloor \frac{m+4}{2} \right \rfloor-1=
       \phantom{-}1 & m\ \text{odd}\ \text{and}\ \beta_{k-1} = \sigma_2^2\beta_k\ \text{or}\ \sigma_1\sigma_2^2\beta_k, \\
       % - \left \lfloor \frac{3(m+2)}{2} \right \rfloor+3 \left \lfloor \frac{m+2}{2} \right \rfloor =
       \phantom{-}0 & $m$\ \text{ even}, \\
        %- \left \lfloor \frac{3(m+2)}{2} \right \rfloor+3 \left \lfloor \frac{m+2}{2} \right \rfloor=
        -1 & \text{otherwise}.
    \end{cases}
\end{align*}

Pending our orientable fillability obstruction in Proposition~\ref{prop:non-fillability} below, the computation of rotation number for $m$ even and strictly less than $-2$ proves Corollary~\ref{cor:rot=0}. 
\subsection{Constructing rulings of $\La_k(\beta)$}
We conclude this section by verifying that $\La_k(\beta)$ is indeed a max-tb Legendrian representative of $\hat{\beta}$.

     \begin{center}
    \begin{figure}[h!]{ \includegraphics[width=0.4\textwidth]{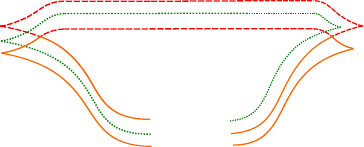}}\caption{Ruling associated to the closure of $\La_0(\beta)$ when $m = -2$.}
    \label{fig: closure_rulings-2}\end{figure}
\end{center}

 \begin{center}
    \begin{figure}[h!]{ \includegraphics[width=0.55\textwidth]{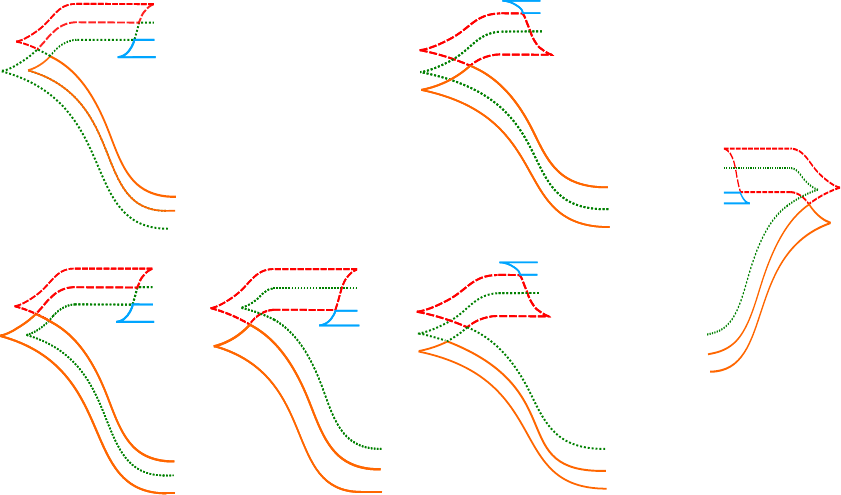}}\caption{Rulings associated to the left and right closures of $\La_k(\beta)$ when $m < -2$ is even. From left to right, the left closures correspond to cases when $\beta_{k-1}$ begins with $\sigma_1^3$ or $\sigma_2\sigma_1^3$, $\sigma_1^2\sigma_2$ or $\sigma_2\sigma_1^2\sigma_2$, and $\sigma_2^2$ or $\sigma_1\sigma_2^2$. Rulings on the top (resp. bottom) row are used if $\beta_k$ begins with $\sigma_1$ (resp. $\sigma_2$).}
    \label{fig: closure_rulings_even}\end{figure}
\end{center}

 \begin{center}
    \begin{figure}[h!]{ \includegraphics[width=\textwidth]{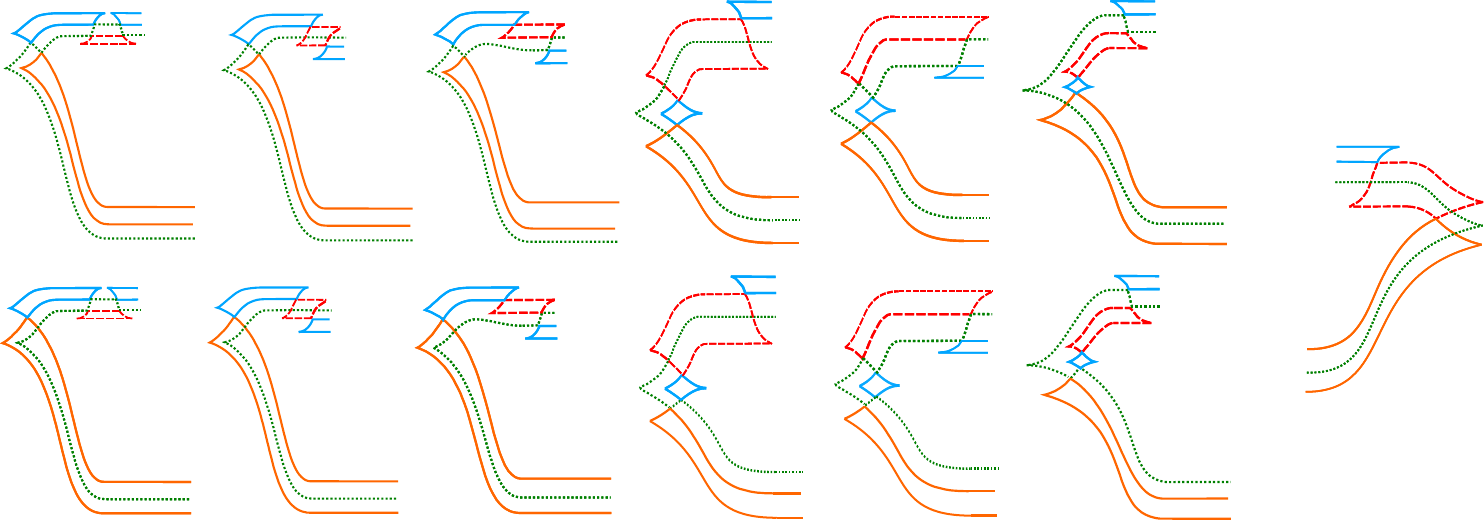}}\caption{Rulings associated to the left and right closures of $\La_k(\beta)$ when $m < -2$ is odd. From left to right, the left closures correspond to cases when $\beta_{k-1}$ begins with $\sigma_1^2$, $\sigma_2\sigma_1^2$, $\sigma_2^2$, and $\sigma_1\sigma_2^2$, where the $\sigma_1^2,\sigma_2^2$ cases are presented with two subcases. Rulings on the top (resp. bottom) row are used if $\beta_k$ begins with $\sigma_1$ (resp. $\sigma_2$).}
    \label{fig: closure_rulings_odd}\end{figure}
\end{center}

\begin{center}
    \begin{figure}[h!]{ \includegraphics[width=.8\textwidth]{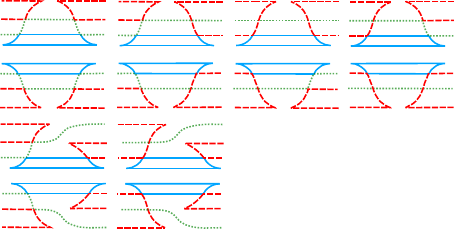}}\caption{Rulings associated to the zigzag patterns on the top of $\La_k(\beta)$. (From top to bottom) When $m$ is even or odd, the zigzag patterns are obtained if $\beta_i$ begins with $\sigma_1^2$ or $\sigma_2^2$,  $\sigma_2^2$ or $\sigma_1^2$, $\sigma_2\sigma_1^2$ or $\sigma_1\sigma_2^2$, and $\sigma_1\sigma_2^2$ or $\sigma_2\sigma_1^2$. The last two columns only appear in the leftmost ($k$th) zigzag.}
    \label{fig: zigzag_rulings}\end{figure}
\end{center}

\begin{proposition}\label{prop: max-tb_ruling}
    For $m \leq -2$, $\La_k(\beta)$ is a max-tb Legendrian representative of $\hat{\beta}$.
\end{proposition}
\begin{proof}
    By Corollary~\ref{cor: Ruling_TB}, it suffices to show that $\La_k(\beta)$ admits a normal ruling. Given that, by construction, $\La_k(\beta)$ can be broken down into the pieces depicted in Figure~\ref{fig: closure_rulings-2}, ~\ref{fig: closure_rulings_odd}, ~\ref{fig: closure_rulings_even}, \ref{fig: zigzag_rulings}, and \ref{fig: braid_rulings}, we specify the ruling patterns locally in these figures and show that they glue without any conflicts. In particular, we use colors to indicate the boundary of the ruling disks and when two pieces of $\La_k(\beta)$ are glued together, the colors of the overlapping strands ought to match with each other. The end result is a ruling consisting of a collection of thin disks on the top of $\La_k(\beta)$ (in red and blue), a thin disk across the braid region of $\La_k(\beta)$ (in orange), and a thick disk that goes around all regions of $\La_k(\beta)$ (in green). See Figure \ref{fig: example_front_ruling} for an example.

    For $m = -2$, we have $\La_k(\beta) = \La_0(\beta)$ and $\beta_k = \beta \sigma_2^{-1}$. The normal ruling is produced by gluing Figure~\ref{fig: closure_rulings-2} and the leftmost braid pattern on the second row of Figure~\ref{fig: braid_rulings} without any zigzags. By assumption, note that $\beta_k$ begins with at least two $\sigma_1$ crossings. See below for a more detailed account of Figure~\ref{fig: braid_rulings}.

Now, assume $m < -2$. We start by gluing the ruling on the top of $\La_k(\beta)$ from right to left. The process largely follows the constructive recipe of $\La_k(\beta)$ by scanning the word $\beta$ from left to right. Begin with the right closure of $\La_k(\beta)$ on which the ruling is depicted in Figure~\ref{fig: closure_rulings_even} when $m$ is even and Figure~\ref{fig: closure_rulings_odd} when $m$ is odd. Since $\beta$ begins with $\sigma_1^2$, note that the first zigzag pattern attached to the right closure is determined to be on the top two rows of Figure~\ref{fig: zigzag_rulings}. It is on the first (resp. second) row if $m$ is even (resp. odd). Provided that Figure~\ref{fig: closure_rulings_even} or ~\ref{fig: closure_rulings_odd} fixes a ruling pattern on the right half of this zigzag, we choose the piece in the second column of Figure~\ref{fig: zigzag_rulings} to extend the ruling pattern leftwards.

The argument above falls under a general inductive paradigm as follows. For $i < k$, consider the $i$-th zigzag pattern obtained from $\beta_{i-1}$ and fix a ruling pattern on its right half that can be found in Figure~\ref{fig: zigzag_rulings}. Our goal is to complete this ruling pattern by choosing one in Figure~\ref{fig: zigzag_rulings} and extend it to the right half of the $(i+1)$-th zigzag. We consider the four cases of $\beta_{i-1}$ outlined before. In Case (1), the $i$-th zigzag pattern is on the top two rows of Figure~\ref{fig: zigzag_rulings}. It is on the first (resp. second) row if $m$ is even (resp. odd). We choose the unique piece in the first two columns that completes the fixed ruling pattern. Since $k_1,\ldots,k_r \ge 2$, note that $\beta_i$ cannot begin with $\sigma_2\sigma_1^2$ and the ruling pattern extends uniquely to the right half of the $(i+1)$-th zigzag that can be found in Figure~\ref{fig: zigzag_rulings}. Case (3) is completely similar. In this case, note that $\beta_{i}$ cannot begin with $\sigma_1\sigma_2^2$ and the ruling pattern extends uniquely. In Case (2), the $i$-th zigzag pattern is in the bottom two rows of Figure~\ref{fig: zigzag_rulings}. It is in the third (resp. fourth) row if $m$ is even (resp. odd). The fixed ruling pattern is the same on the right half of the two candidates. If $\beta_i$ begins with $\sigma_1^2$ or $\sigma_1\sigma_2^2$ (resp. $\sigma_2^2$), we choose the ruling pattern on the left (resp. right) and it extends uniquely to the right half of the $(i+1)$-th zigzag; $\beta_i$ cannot begin with $\sigma_2\sigma_1^2$. Case (4) is completely similar to Case (2) in the sense of Case (3) being similar to Case (1). Throughout this process, note that we have only used the first two columns of Figure~\ref{fig: zigzag_rulings}.

Proceeding inductively, we obtain a ruling pattern on the right half of the $k$-th zigzag among the ones depicted in Figure~\ref{fig: zigzag_rulings}. In the mean time, Figure~\ref{fig: closure_rulings_even} or ~\ref{fig: closure_rulings_odd} fixes a ruling pattern on the left half of the $k$-th zigzag. For the left closures with two ruling patterns presented, we choose the one on top if $\beta_k$ begins with $\sigma_1$ and the one on bottom otherwise. Note that when $m$ is even and $\beta_{k-1}$ begins with $\sigma_1^2\sigma_2$ or $\sigma_2\sigma_1^2\sigma_2$, $\beta_k$ necessarily begins with $\sigma_2$. By inspection, we see that the left closures extend the ruling pattern on the right half of the $k$-th zigzag in all cases without any conflicts. The resulting ruling pattern on the $k$-th zigzag is among those depicted in Figure~\ref{fig: zigzag_rulings}.

\begin{center}
    \begin{figure}[h!]{ \includegraphics[width=\textwidth]{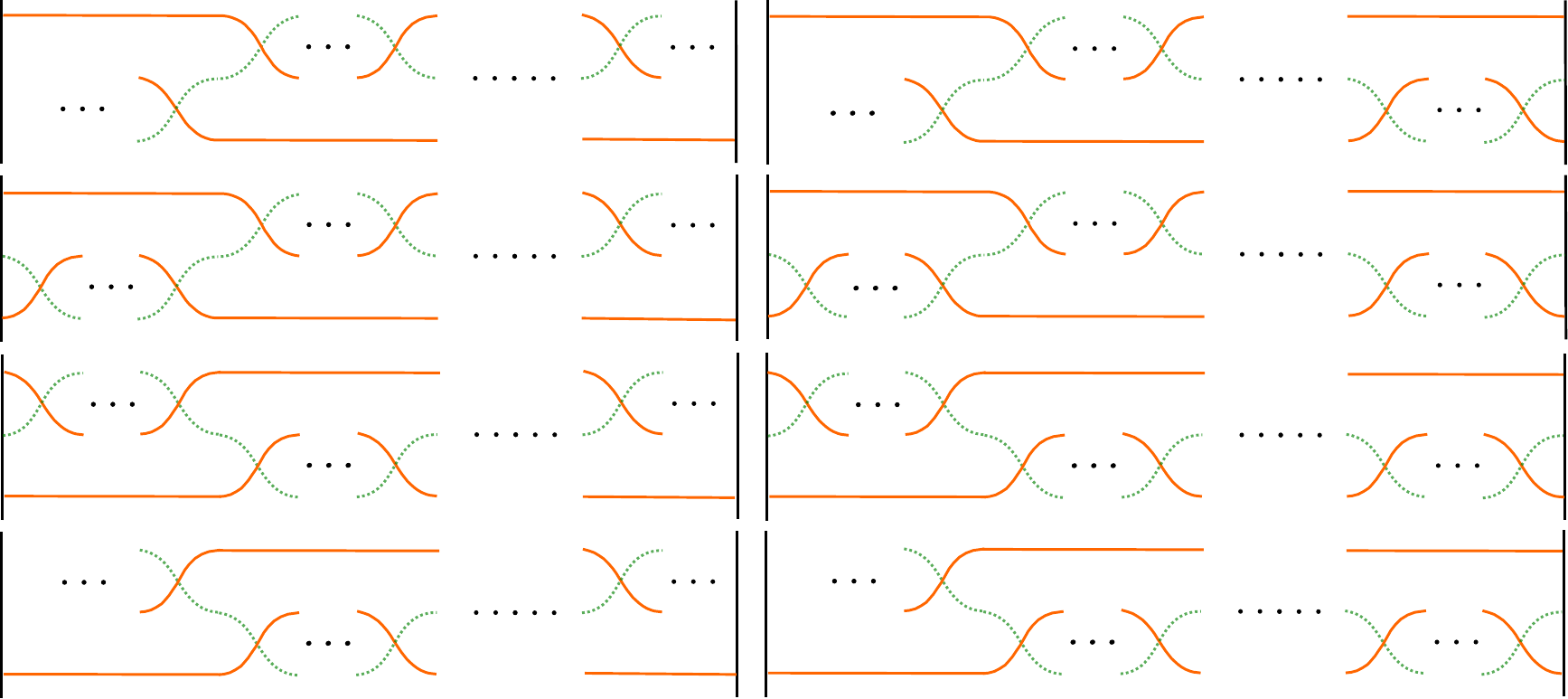}}\caption{Rulings associated to the braid region of $\La_k(\beta)$ when $m$ is even (left) and odd (right). The ellipses within each block $\sigma_i^{k_i}$ indicate switches.}
    \label{fig: braid_rulings}\end{figure}
\end{center}

It remains to fit a ruling pattern of the braid region $\beta_k$ into the picture. Such a pattern is given in Figure~\ref{fig: braid_rulings} for all possible configurations of $\beta_k$, consisting of only a green path and an orange path. Unless otherwise stated in the figure, in each block $\sigma_i^{k_i}$ of $\beta_k$, all crossings except the first and the last ones are switches indicated by an ellipsis. However, note that the first and last block of $\beta_k$ may begin or end with a number of switches respectively. Again, an inspection verifies that the ruling patterns agree with those on the closures in Figure~\ref{fig: closure_rulings_even} or ~\ref{fig: closure_rulings_odd} in all cases. In particular, when the green path enters (resp. exists) in the middle of the two orange paths, we are ensured that the first (resp. last) block of $\beta_k$ contains at least two crossings.

In summary, we obtain a collection of embedded disks with boundary on $\La_k(\beta)$, as claimed above. Further, all switches are disjoint except those at $\sigma_2$ crossings in the middle of $\beta_k$, which are nested. Thus, we obtain a normal ruling of $\La_k(\beta)$ and by Corollary~\ref{cor: Ruling_TB}, this proves that $\La_k(\beta)$ is a max-tb representative of its smooth link type $\hat{\beta}$.
\end{proof}

\section{Orientable Fillability of $3$-Braid Closures}
\label{sec: orientable-fill}

We are now ready to discuss exact Lagrangian fillability of (quasipositive) $3$-braid closures $\hat{\beta}$. In this section, we complete the proof of Theorem~\ref{thm: fillability} by showing that $\hat{\beta}$ is orientably fillable if $m = -2$ and not if $m < -2$. We begin with the construction of orientable fillings.

\begin{proposition}
    \label{prop:m=-2_fillable}
    $\hat{\beta}$ is orientably exact Lagrangian fillable if $m = -2$.
\end{proposition}

\begin{proof}
It suffices to construct an orientable exact Lagrangian filling of $\La_0(\beta)$. In the notation of Equation~\ref{eq:garside-normal}, note that according to Orevkov's criterion of quasipositivity, $r \ge 4$ when $m = -2$. Indeed, one can determine by direct inspection of braids equivalent to $\Delta^2$ that no braid of the form $\sigma_1^{k_1}\sigma_2^{k_2}\sigma_1^{k_1}$ can be reduced to a braid equivalent to $\Delta^2$ by removing any number of crossings. In the case of $r=4$, we observe that when $\beta=\sigma_1^{k_1}\sigma_2^{k_2}\sigma_1^{k_3}\sigma_2^{k_4}\Delta^{-2}$, the Legendrian $\La_0(\beta)$ admits an exact Lagrangian filling constructed by applying the local model in Figure~\ref{fig: D4-} to pinch the crossings numbered $3, \dots, k_1$ in the first block, $2, \dots, k_2$ in the second block, $3, \dots, k_3$ in the third block, and $2, \dots, k_4$ in the last block. Applying a Reidemeister III move to the resulting $\sigma_1^2\sigma_2\sigma_1^2\sigma_2$ results in $\La_0(\Delta^2\Delta^{-2})$, which is isotopic to three unlinked max-tb Legendrian unknots. %This observation can be verified by pinching the fourth and eighth crossings of  $\beta$ and applying a braid move to obtain $\La_0(\Delta^2\Delta^{-2})$, which is isotopic to three unlinked max-tb Legendrian unknots. 
For $r> 4$, we can construct a filling by first performing the sequence of pinchings as in the $r=4$ case, yielding a cobordism to $\La_0(\Delta^2\sigma_1^{k_5}\dots \sigma_2^{k_r}\Delta^{-2}).$ We then inductively eliminate the $k_i$th block, starting with $i=5$ and proceeding in order to $i=r$, by performing the following steps: first, use braid equivalences to write $\Delta^2=\Delta\sigma_1\sigma_2\sigma_1$ when $i$ is odd and $\Delta^2=\Delta\sigma_2\sigma_1\sigma_2$ when $i$ is even. Then, pinch all of the crossings in the $k_i$th block, using the rightmost crossing in $\Delta^2$ as the left crossing in the local model of Figure~\ref{fig: D4-}. After inductively eliminating each of the blocks in this fashion, we again arrive at $\La_0(\Delta^2\Delta^{-2})$, which we fill in with exact Lagrangian disks.
\end{proof}

Next, we turn to the obstructive part of Theorem \ref{thm: fillability}, namely,

\begin{proposition}\label{prop:non-fillability}
    $\hat{\beta}$ is not orientably exact Lagrangian fillable if $m < -2$.
\end{proposition}

\begin{proof}
    Using the max-tb Legendrian representatives $\La_k(\beta)$ of $\hat{\beta}$ found in Section~\ref{sec:max-tb}, we prove explicitly that the HOMFLY bound on $\overline{\tb}(\hat{\beta})$ is not sharp, which obstructs orientable fillability by Lemma \ref{lem:orientable-obs}. Recall from our earlier computations in Section~\ref{ssec:links_background} and Section~\ref{sec:max-tb} that
    \begin{align*}
        \tb(\La_k(\beta)) - \tb(\La_0(\beta))  = - \left \lfloor \frac{m+2}{2} \right \rfloor, \\
        |\rot(\La_k(\beta))| - |\rot(\La_0(\beta))|  \leq 3\left \lfloor \frac{m+2}{2} \right \rfloor,
    \end{align*}
    Thus, for $m < -2$,
    \begin{equation*}
    \overline{\tb}(\hat{\beta}) = \tb(\La_k(\beta)) \leq \tb(\La_k(\beta)) + |\rot(\La_k(\beta))| <  \tb(\La_0(\beta)) + |\rot(\La_0(\beta)| \leq d_{P_K},
    \end{equation*}
    and the HOMFLY bound on $\overline{\tb}(\hat{\beta})$ is not sharp, as desired.
\end{proof}

\begin{remark}
    Rather than computing the HOMFLY bound directly, one could also observe that for $m<-2$, the Legendrian $\La_k(\beta)$ admits no orientable rulings as each ruling must have a switch at at least one of the negative crossings in the left or right closure. The desired result then follows from the fact that orientably fillable Legendrian links must admit an orientable ruling.
\end{remark}

We conclude this section by a comment about our orientable filling construction from the Legendrian weave perspective, which we then use to conclude Corollary~ \ref{cor: weaves}. We refer readers to \cite[Section 2]{CasalsZaslow} for necessary preliminaries on Legendrian weaves. 

\begin{corollary}
$\La_0(\beta\Delta^{m})$ admits an exact Lagrangian filling obtained as the Lagrangian projection of a Legendrian weave for all $m\geq -2$. 
\end{corollary}

\begin{proof}
    The exact Lagrangian fillings of Proposition \ref{prop:m=-2_fillable} are constructed by using braid isotopies and pinching cobordisms to produce an orientable exact Lagrangian cobordism to the unlink. In the language of \cite[Section 2]{CasalsZaslow}, these braid isotopies correspond to hexavalent vertices, while the pinching cobordisms correspond to trivalent vertices. For $m\geq -1$, the orientable exact Lagrangian fillings of \cite{CGGLSS} discussed in Subsection~\ref{ssec:background-fillings} are constructed as the Lagrangian projection of Legendrian weaves. 
\end{proof}

\section{Non-orientable Fillability of Quasipositive $3$-Braid Closures}
\label{sec:non-orientable-fill}

In this section, we characterize non-orientable fillability of quasipositive 3-braid closures, proving Theorem~\ref{thm: intro_non-orientable_fillability}. First, we extend Theorem~\ref{thm:ccpr+} from braid positive Legendrians to the case of $m=-1$ and prove the obstructive part of Theorem~\ref{thm: intro_non-orientable_fillability}. We do so following the proof technique in \cite{CCPRSY}, namely showing that every normal ruling of $\La_0(\beta)$ is orientable.

\begin{lemma}\label{lem: rulings_positive}

Every normal ruling of $\La_0(\beta_+\Delta^{-1})$ is orientable. 
\end{lemma}

\begin{proof}
    $\La_0(\beta_+\Delta^{-1})$ has one negative half twist on the right of the front diagram and no negative crossings on the left. This guarantees that all of the three ruling disks must be thick disks passing through both the top and bottom set of strands of the front projection. Inspecting the ruling disks in order of the positions of their left cusps, from the outermost toward the innermost, we see that the first disk must end at the top cusp on the right without switches at any negative crossings, as the behavior at the leftmost cusp forces the lower half of the disk to pass through the braid. Repeating this argument iteratively, it follows that the second and third disk must end at the middle and bottom cusp on the right, respectively, without switches at any negative crossings.
\end{proof}

Note that the above result and method of proof hold for braid index greater than three as well. 

\begin{proposition}
    $\hat{\beta}$ is not non-orientably decomposably fillable if $m = -1$. 
\end{proposition}

\begin{proof}
    This follows immediately from Lemma~\ref{lem:  rulings_positive} and Lemma~\ref{lem: non-orientable obstruction}.
\end{proof}

Finally, we construct a non-orientable exact Lagrangian filling of $\hat{\beta}$ for $m \leq -2$.

\begin{proposition}\label{prop: non-orientably fillable}
$\hat{\beta}$ is non-orientably fillable if $m \leq -2$. 
\end{proposition}

\begin{proof}
We prove the statement above by showing that the ruling constructed in the proof of Proposition \ref{prop: max-tb_ruling} is induced by a non-orientable filling. Specifically, we claim that the switches in the ruling can all be resolved via the pinching cobordism depicted in Figure~\ref{fig: D4-} or, in the presence of a cusp, by first performing a Reidemeister II move, as in Figure \ref{fig: D4-RII} and then applying the local move from Figure~\ref{fig: D4-}.

We first verify the claim for $m=-2$. In this case, recall that the ruling of $\La_{k=0}(\beta)$ is given by combining Figure~\ref{fig: closure_rulings-2} with the braid region depicted in the first column, second row of Figure~\ref{fig: braid_rulings}. To resolve the switched crossings, we start with the rightmost switch of Figure~\ref{fig: closure_rulings-2}, and resolve it via a horizontal reflection of Figures~\ref{fig: D4-RII} and \ref{fig: D4-}. This then allows us to resolve all of the switches in the braid region again starting with the rightmost, freely performing Reidemeister II moves with the orange strands as needed. After resolving all of these switches, we can perform one more Reidemeister II move to make the green and orange ruling disks disjoint, allowing us to resolve the switch between the red and orange ruling disks on the left of Figure~\ref{fig: braid_rulings}. This yields three unlinked unknots, which we can then fill in with disks to produce an exact Lagrangian filling of $\La_k(\beta)$.

Our claim can be verified for $m<-2$ by an inspection of the local pieces of rulings depicted in Figures~\ref{fig: closure_rulings_even}, \ref{fig: closure_rulings_odd}, \ref{fig: zigzag_rulings}, and \ref{fig: braid_rulings}. In particular, we see that during the construction of the ruling in the proof of Proposition \ref{prop: max-tb_ruling}, the red and blue thin ruling disks are either disjoint from the rest of the front after resolving the switches, or they can be made disjoint after performing a single Reidemeister II move to pull them away from the green strand. Note that we require gluing of zig-zag patterns exactly as specified in the proof, as otherwise we might have the green strand non-trivially linked with one of the red or blue thin disks. Likewise, when we examine the green and orange disks of the ruling, we can produce a pair of max-tb unknots by a sequence of Reidemeister II moves and crossing resolutions, as in the $m=-2$ case. In particular, the ruling patterns in Figure~\ref{fig: braid_rulings} where the first crossing of $\beta$ is switched are matched with ruling patterns from Figures~\ref{fig: closure_rulings_even} and \ref{fig: closure_rulings_odd} that force the corresponding ruling disks to be disjoint at the first crossing of $\beta$. 

To produce an exact Lagrangian filling inducing the ruling described in the proof of Proposition \ref{prop: max-tb_ruling}, we start scanning the front left to right and performing pinching cobordisms at each of the switched crossings as they appear, freely applying the isotopy from Figure \ref{fig: D4-RII} to facilitate this. The one obstruction to this particular ordering of pinching cobordisms that we encounter is when, in the zigzag patterns at the top of the front, a red ruling disk has a green strand passing through it. The red and green ruling disks are unlinked, but we cannot isotope them to be disjoint without first resolving one of the switched crossings of the red disk. When the red disk appears in a position similar to the leftmost red disk in Figure \ref{fig: example_front_ruling} or its vertical reflection, we instead perform the pinching cobordism at the other switched crossing involving the red disk. We then isotope the red disk to be disjoint from the green, and performing a pinching cobordism at the remaining switched crossing involving the red disk. When resolving the switched crossings in $\beta$, we freely perform Reidemeister II moves as before. Direct inspection of the rulings associated to the left and right closures as pictured in Figure~\ref{fig: closure_rulings_even} and \ref{fig: closure_rulings_odd} verifies that we can perform a similar process there as well.
The end result of this process is a collection of unlinked max-tb unknots, which we can fill in with the minimum cobordism pictured in Figure \ref{fig: minima_saddle} (left) to produce a non-orientable exact Lagrangian filling of $\La_k(\beta)$.
\end{proof}

\begin{center}
    \begin{figure}[h!]{ \includegraphics[width=.3\textwidth]{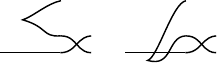}}\caption{Preparatory isotopy before performing a pinching cobordism.}
    \label{fig: D4-RII}\end{figure}
\end{center}
\bibliographystyle{alpha}
\bibliography{main.bib}

\end{document}